\documentclass[12pt,reqno]{article}
\usepackage{amsmath,amsthm,amssymb,cite}
\newtheorem{theorem}{Theorem}[section]
    \newtheorem{proposition}[theorem]{Proposition}
    \newtheorem{lemma}[theorem]{Lemma}
    
    \newtheorem{remark}[theorem]{Remark}
    \newtheorem{defn}[theorem]{Definition}
    \numberwithin{equation}{section}
    \numberwithin{theorem}{section}
    
    \renewcommand{\(}{\left(} \renewcommand{\)}{\right)}
    
    \newcommand{\noi}{\noindent}\newcommand{\nn}{\nonumber}
    \renewcommand{\epsilon}{\varepsilon}

    \newcommand{\e}{\mbox{e}}
    
\begin{document}
    
    \title {The asymptotic expansion for $n!$ and Lagrange inversion formula.}
    \author {Stella Brassesco \and
     Miguel A. M\'endez}
     \date{}
     \maketitle
     
    \begin{center} { Departamento de Matem\'aticas,
    Instituto Venezolano de Investigaciones Cient\'{\i}ficas, Apartado Postal
    20632 Caracas
    1020--A, Venezuela }
    \end{center}
    \begin{center}
    \texttt{sbrasses@ivic.ve},
    \texttt {mmendez@ivic.ve  }
    \end{center}
    \begin{abstract}
    We obtain an explicit simple  formula for the coefficients
    of the  asymptotic expansion for the factorial of a natural number,
    \begin{equation}
    \nonumber n!=  n^n\sqrt{2\pi n}\,\mbox{e}^{-n}\big\{
    1\,+\,\frac{a_1}{n}\,+\,
    \frac{a_2}{n^2}+\,\frac{a_3}{n^3}+\cdots\big\},
    \end{equation}
    in terms of
    derivatives of powers of an elementary function that we call
     {\it  normalized left truncated exponential function}.
    The unique explicit expression for the ${a_k}$  that appears to be known
    is that of
    Comtet in \cite{Comtet}, which is given in terms of sums of
    associated Stirling numbers of the first kind.
    By considering the bivariate generating function of
     the associated Stirling numbers of the second kind,
      another expression for the coefficients in terms of them
       follows also from our
    analysis.
    Comparison with Comtet's  expression
     yields an
    identity which is somehow unexpected if
     considering the
     combinatorial
    meaning of the terms. It suggests by analogy  another possible formula
     for the
    coefficients, in terms of a {\it normalized left truncated logarithm},
    that in fact
     proves to be true.  The resulting coefficients, as well as the
     first ones are identified
     via the
    Lagrange inversion formula as the odd coefficients of the inverse of a
     pair of formal series. This in particular  leads to  the identification
     of a couple of
     simple implicit
    equations, which permits us to obtain also some  recurrences related to the
     ${a_{k}}'s$.
    \end{abstract}
      \section {Introduction}
      Consider, for $t>0$
    $$
    \e^t\,=\,\sum_{j\ge 0}\frac{t^j}{j!} \quad \mbox{ or }\quad
     \sum_{j\ge 0}e^{-t}\,\frac{t^j}{j!}\,=1,
     $$
    which means that we may think $\e^{-t}\,\frac{t^j}{j!}$ as the probability
    $p_j$ that a
    random  variable $X_t$ takes the value $j$. This amounts to say that $X_t$
    has a Poisson
    distribution with parameter $t$, whose expectation, variance and
    characteristic function,
    $\varphi_{\!_{X_t}}(\theta)$ for
     $\theta \in \mathbb R $ are:
    \begin{align}
    \nonumber
    &E(X_t)=t\qquad Var(X_t)=t\\
    \label{ftrans} &\varphi_{\!_{X_t}}(\theta):=\,
    E\big(\e^{i\theta\,X_t}\big)\,=\,\sum_{j\ge 0}\e^{i\theta j}
     p_j= \sum_{j\ge 0}\e^{-t}
    \frac{(\e^{i\theta }t)^j}{j!}\,= \,\e^{t\,(\e^{i\theta}-1)}
    \end{align}
    Take now $t=n\in\mathbb N$, and consider the random variable
    \begin{equation}
    \label{zn} Z_n\,=\,\frac{X_n-n}{\sqrt{n}}.
    \end{equation}
    From the Central Limit Theorem we know that the characteristic
     function of $Z_n$
    converges pointwise to that of a standard Normal random variable:
     for
     $\theta\in \mathbb R$,
    \begin{equation}
    \label{clt}
     E(\e^{i\theta\,Z_n})
    \rightarrow\e^{-\theta^2/2}\quad\mbox{ as } n\to \infty.
    \end{equation}
    On the other hand, projecting the series \eqref{ftrans} with
     $t=n$ over $\e^{-i\theta n}$
    in $L^2[-\pi,\pi]$ we obtain
    $$\int^{\pi}_{-\pi}\!d\theta\,E\big(\e^{i\theta\,X_n}\big)\e^{-i\theta n}=
    \int^{\pi}_{-\pi}\!d\theta\,\frac{\e^{-n}\,n^n}{n!}\,= \,2\pi
    \frac{\e^{-n}\,n^n}{n!}
    $$
    Recalling the definition \eqref{zn} of $Z_n$, a change of variables in the
    first integral
    above
     yields
    $$
    \int^{\pi\sqrt{n}}_{-\pi\sqrt{n}}\!d\theta\, E\big(\e^{i\theta \,Z_n
    }\big)\,=\,
    \int^{\pi\sqrt{n}}_{-\pi\sqrt{n}}\!d\theta\, E\big(\e^{i\frac{\theta}{\sqrt
    n}(X_n-n)}\big) \,=\,2\pi\,\sqrt{n}\,\frac{\e^{-n}\,n^n}{n!},
    $$
    so
    \begin{equation}
    \label{n!1}
     n!=\frac{\sqrt{2\pi n}\,
    \e^{-n}\,n^n}{\frac{1}{\sqrt{2\pi}} \int^{\pi\sqrt{n}}_{-\pi\sqrt{n}}
    d\theta\,
    E\big(\e^{i\theta \,Z_n }\big)}
    \end{equation}
    If the convergence in \eqref{clt} could be seen to hold uniformly, we would
    recover
    Stirling's formula:
    \begin{equation}
    \nn
     n!\approx \sqrt{2\pi n}\,\e^{-n}n^n
    \end{equation}
    It is not difficult to see that this is indeed the case, as
     has been noticed by several
    authors, see for instance \cite {pinsky} and the references therein,  or
    Sect 27,
    exercise 18 of \cite {Billingsley}. L. B\'aez--Duarte considered
    in  \cite{bd}
    the above method
     to obtain the asymptotics for the general term
     of a convergent power series of a function $f$  with non--negative terms,
     under conditions  that guarantee that a uniform Central Limit Theorem
     for the corresponding random variables holds and, as a particular instance,
     he obtains the asymptotics for
     $p(n)$, the number of partitions of an integer $n$. The approach
     has a longer history, and  goes back to the articles by W.K. Hayman
     \cite{H} and
     E.A. Bender \cite{B}.
     
     Here, we obtain an asymptotic expansion for the denominator in \eqref{n!1}
     in terms of inverse powers of $n$, and show that it
     can be explicitly   inverted, yielding a simple expression for
     the coefficients $a_k$  of the expansion of $n!$,
      in terms of an elementary function
     (see Theorem \ref {t1} below).
     To our knowledge, the unique explicit formula  for those coefficients
     is  that given by Comtet as an exercise in page 267 of \cite{Comtet},
     in terms of sums of derangements $d_3$ of
     integers with certain restrictions, related with the
      Stirling numbers of the first kind.
     As a corollary of the proof of our first formula,
      after recalling the bivariate
     generating function for the associated Stirling number
     of the second kind $S_3$,
     we obtain a second  expression for the coefficients,
      which is exactly alike  to that  in
     \cite{Comtet}, and
     in particular yields an   unexpected
      equality  between certain alternating sums of  associated
       Stirling numbers
       of the first and  second kind. This identity suggested
        another  expression
       for the $a_k$,
       in terms of a {\it normalized left truncated} logarithm
        function corresponding
        exactly with the first
       one,
        that we proved to be true (see Theorem \ref {t3}).
       Finally, we recognize in these two expressions the odd terms of the
        series obtained by  the  Lagrange inversion formula
        applied to appropriate
        functions. Each of these two inverses satisfy a simple
        implicit equation.
       By differentiation  we obtain a couple of differential equations
        that lead to a pair of recurrent relations
         that could be used also
        to generate the ${a_{k}}'s$. One of them is similar to that obtained 
        in
        \cite{Marsaglia} using other techniques (see also \cite{Corless}).
        
    Although we state the results as expressions for the
    coefficients of the expansion of  $n!$ in inverse powers of $n$,
     they could be stated as well
    as expressions for the coefficients of the
     asymptotic expansion of the gamma
    function  in inverse powers of $z\in \mathbb C$,
     for $|z|\to \infty$ in the region $\arg(z)\in [0,\pi)$. This  follows
    at once  from the existence of this last expansion. For this and
    other properties
    of  the gamma, see \cite{aar}.
      Let us also
      note that the first few coefficients, that in particular  appear
      in most books and tables of
       special functions (for instance \cite {GR} or \cite {Olver}),
       can be computed one by one by exponentiating the well known
        Stirling series
       for the logarithm of the gamma, or from recurrence relations
        obtained by different
       techniques (as in \cite {W} or  \cite {Marsaglia}). We think however
       that having explicit and  simple formulae has an intrinsic interest.
       In particular, we  show that they render some combinatorial identities, and
       relate with the  Lagrange inversion formula for elementary functions,
        what  helps to underestand  the relation of the expansion with
        the Lambert  W function appearing in \cite{Corless}.

    We provide precise statements  concerning the asymptotic
    expansions, and obtain the formulae for the coefficients in the next
    section.
    In the last  section, we show the
     relationship with the Lagrange inversion formula, obtain  recursive
    formulae and generalise the identities obtained between sums of $S_3$ and
    $d_3$.

      \section{ The asymptotic expansion and formulae for the coefficients }
     
      Let us denote by $\partial^k f$ the $k$th
      derivative of a function $f$,
      with respect to its  real or complex variable.
      Our first result is  the following  formula for the coefficients of the
      expansion
      of $n!$ in powers of $\frac 1n$.
      \begin{theorem}
      \label{t1}
      Let $n\in  \mathbb N$. Then the coefficients of the expansion
      \begin{equation}
      \label{asym}
       n!\asymp\sqrt{2\pi n }\,\e^{-n} n^n \Big(
     \,\sum_{k\ge 0} \frac{1}{n^k}\,a_k \Big)\quad
     \end{equation}
     \begin{align}
      \label{ak}
     &\mbox{are given by } \quad
     a_k= \frac{1}{2^k\,k!}\,\partial^{2k} (G^{-\frac{2k+1}{2}})(0)
     \\ \nn
     &\mbox {where }\qquad
       G(x)\,=\,2\,\frac{\e^x-1-x}{x^2}\,=\, 2\sum_{j\ge 0}\frac{x^j}{(j+2)!}
     \end{align}
      \end{theorem}
      \begin{remark}
      \label{re}
      Formula \eqref{asym} is to be understood as an asymptotic
      expansion, that is, for
      any given $N\ge 0$,
      $$ n!= \sqrt{2\pi n }\,\e^{-n} n^n \Big(
     \,\sum_{k=0}^{N} \frac{1}{n^k}\,a_k \, +\,R_{N+1}\Big)
      $$
      with $
      R_{N+1}=O(\frac{1}{n^{N+1}})$ as $n\to \infty$. See for instance
      \cite {Olver}
      or \cite{WW} on the subject of asymptotic expansions.
      \end{remark}
      The function $G$ above is what we call a
       {\it  normalized left truncated exponential}. In general we define
       \begin{defn}
       \label {def}
       Given a function $F(x)=\sum f_n\frac{x^n}{n!}$ with $f_2\neq 0$,
       its  {\it  normalized left truncated}
       associated function $F_2$ is defined as
       $$
       F_2(x)=\frac{F(x)-f_0-f_1\,x}{f_2\,x^2/2}
       $$
       \end{defn}
      We also obtain another expression for $a_k$ in terms of the
       $3-$associated Stirling numbers of second kind.
      Recall that, for $r$ an integer $\ge1$, $n,k$ integers $\ge0$,
      the $r-$associated Stirling number of the  second kind $S_r(n,k)$
      is the number of
      partitions of a set of $n$ elements into $k$ blocks, all with at least
      $r$ elements. The convention is that $S_r(0,0)=1$.  Then we have
      \begin{theorem}
      \label {t2}
     The coefficients $a_k$ in \eqref{asym} are also given
     in terms of the $3$--associated Stirling numbers of second kind by
      \begin{equation}
      \label{c2}
       a_k\,=\,\sum_{j= 0}^{2k} (-1)^j\,
       \frac{\,S_3\big(2(j+k),j\big)}{2^{j+k}\,(j+k)!}
       \end{equation}
      \end{theorem}
      \vskip .5 cm
      \noi{\it Proof of Theorem \ref{t1}}.
      The strategy will consist in expanding the denominator in \eqref{n!1}
      in powers of $\frac 1n$,
      and then taking the inverse of the resulting series.
      From \eqref{ftrans} and \eqref{zn}, for $\theta \in \mathbb R$,
      \begin{equation}
      \label{ig}
      E\big(\e^{i\theta\,Z_n}\big)\,=\,
    \e^{n\,\big(\e^{\frac{i\theta}{\sqrt n}}-1-\frac{i\theta}{\sqrt
    n}\big)}\,=\,
     \e^{-\frac{\theta^2}{2}-\theta^2g(\frac{i\theta}{\sqrt{n}})},
    \end{equation}
    where, for $z\in \mathbb C$
    \begin{equation}
    \label{g} g(z):=\frac{\e^z-1-z-z^2/2}{z^2}\,=\,\sum _{k\ge
    1}\frac{z^k}{(k+2)!}.
    \end{equation}
    For given  $K\in \mathbb N$, consider a Taylor expansion of order $K$
     in powers of $z$
    of $\e^{-\theta^2 g(z)}$:
     \begin{equation}
    \label{te} \e^{-\theta^2\,g(z)}=\sum_{j=0}^{K}\frac{z^j}{j!}\,
    \partial^j\big(\e^{-\theta^2\,g(z)}\big)\big\vert_{z=0}\,+\,R_{K+1}(z)
    \end{equation}
    The remainder  $R_{K+1}$ satisfies the estimate
    \begin{equation}
    \label{rest} \big\vert R_{K+1}(z) \big\vert\le \frac{|z|^{K+1}}{(K+1)!}\,
    \sup_{\zeta\in[0,z]}\big\vert\,
    \partial^{K+1}\big(\e^{-\theta^2\,g(\zeta)}\big)\big\vert
    \end{equation}
    Substitution of \eqref{ig} and \eqref{te} in the denominator
     $D_n$ in \eqref{n!1} yields
    \begin{align}
    \nonumber
    D_n:=&\frac{1}{\sqrt{2\pi}}\,\int^{\pi\sqrt{n}}_{-\pi\sqrt{n}}\!d\theta\,
    E\big(\e^{i\theta \,Z_n }\big)=\frac{1}{\sqrt{2\pi}}\,
    \int^{\pi\sqrt{n}}_{-\pi\sqrt{n}}\!d\theta\,
    \e^{-\frac{\theta^2}{2}}\e^{-\theta^2\,g(\frac{i\theta}{\sqrt n})}
    \\
    \nonumber =& \frac{1}{\sqrt{2\pi}}\, \sum_{j=0}^{K}\,
    \int^{\pi\sqrt{n}}_{-\pi\sqrt{n}}\!d\theta\,
    \e^{-\frac{\theta^2}{2}}\big(\frac{i\theta}{\sqrt{n}}\big)^j
    \frac{1}{j!}\,\partial^j\Big(\e^{-\theta^2\,g(z)}\Big)\Big|_{z=0}
    \\
     \label{in}
    &\qquad \qquad +
    \frac{1}{\sqrt{2\pi}}\,\int^{\pi\sqrt{n}}_{-\pi\sqrt{n}}\!d\theta\,
    \e^{-\frac{\theta^2}{2}}\,R_{K+1}(\frac{i\theta}{\sqrt n}).
    \end{align}
    It can be easily seen by induction that, for $j\in \mathbb N$ and $z\in
    \mathbb C$
    $$
    \partial^j\,\big(\e^{-\theta^2\,g(z)}\big)=\e^{-\theta^2g(z)}
    P_{j,z}(\theta^2)
    $$
    where $P_{j,y}(\cdot) $ is a polynomial of degree $j$ whose coefficients
    depend on
    products of derivatives
     of $g$ of order up to $j$ evaluated in $z$.
     From \eqref{g}
     it is easy to see that $g$ and all its derivatives are
     bounded over compacts. Moreover, if we denote by
     $\mathcal R (z)$ the real part of $z$, we have
    \begin{equation}
    \label{expg}\big\vert \e^{-\theta^2\,g(\frac{i\theta}{\sqrt
    n})}\big\vert\,\le
    \e^{-\theta^2 \mathcal R\big( g (\frac{i\theta}{\sqrt n})\big)}
    \end{equation}
    and, for $x\in \mathbb R$,
    \begin{align}
    \nonumber \mathcal R\big(g(ix)\big)=\,&\mathcal R \left(\frac{\e^{ix}
    -1-ix+x^2/2}{-x^2}\right)
    =-\left(\frac {\cos x-1+x^2/2}{x^2}\right)\\
    \label{real} &= -\frac{x^2}{4!}\,\partial ^4\cos (ax)\quad \mbox { for some
    } a\in [0,1]
    \end{align}
    From \eqref{rest}, \eqref{expg} and \eqref{real}, we can estimate the last
    integral in
    \eqref{in} as follows:
     \begin{multline}
     \label{rk}
     \Big\vert \int^{\pi\sqrt{n}}_{-\pi\sqrt{n}}\!d\theta\,
    \e^{-\frac{\theta^2}{2}}\,R_K(\frac{i\theta}{\sqrt n})\Big\vert\\
    \le
     \frac{1}{(K+1)!}\int^{\pi\sqrt{n}}_{-\pi\sqrt{n}}\!d\theta\,
    \e^{-\frac{\theta^2}{2}}\,\big\vert\frac{i\theta}{\sqrt n}\big\vert ^{K+1}
    e^{\theta^2\frac{\pi^2}{4!}} \sup_{|y|\le\pi} \big\vert P_{K+1,y}(\theta^2)
    \big\vert \le
    \frac{C_{K}}{n^{\frac{K+1}{2}}},
    \end{multline}
    where $C_{K}$ is a constant which depends on $K$ but not on $n$. Then,
    observe that the
    difference of each integral in the sum in \eqref{in} with respect to the
    integral in the
    whole $\mathbb R$ is exponentially small in $n$: for any $j\ge 0$, there is
    some $\alpha$
    positive,
    \begin{multline}
    \label{e-n} \Big|\int^{\pi\sqrt{n}}_{-\pi\sqrt{n}}\!d\theta\,
    \e^{-\frac{\theta^2}{2}}\big(\frac{i\theta}{\sqrt{n}}\big)^j
    \frac{1}{j!}\,\partial^{j}\Big(\e^{-\theta^2\,g(x)}\Big)\Big|_{x=0}
    \\\,-\,
    \int^{\infty}_{-\infty}\!d\theta\,
    \e^{-\frac{\theta^2}{2}}\big(\frac{i\theta}{\sqrt{n}}\big)^j
    \frac{1}{j!}\,\partial^{j}\Big(\e^{-\theta^2\,g(x)}\Big)\Big|_{x=0}\Big|
    \\
    \le \int_{|\theta|>\pi\sqrt{n}}\!d\theta\,
    \e^{-\frac{\theta^2}{2}}\Big|\big(\frac{i\theta}{\sqrt{n}}\big)^j
    \frac{1}{j!}\,
    P_{j,0}(\theta^2) \Big|\,\le\,\e^{-\alpha n}
    \end{multline}
    Moreover, if we interchange the derivative with the integral in the second
    line above,
    and then compute the integral, we obtain:
    \begin{multline}
    \frac{1}{\sqrt \pi}\int^{\infty}_{-\infty}\!d\theta\,
    \e^{-\frac{\theta^2}{2}}\big(\frac{i\theta}{\sqrt{n}}\big)^j
    \frac{1}{j!}\,\partial^{j}\Big(\e^{-\theta^2\,g(x)}\Big)\Big|_{x=0}\,=
    \\ \label{coef}
    \begin{cases}
    \qquad 0\quad  &\mbox{ if $j$ odd}\\
    \big(\frac{i}{\sqrt{n}}\big)^j\,\frac{1}{j!}\,(j-1)\,!!\,
    \partial ^j\Big(\big( 1+2g(x)\big)^{-\frac{j+1}{2}}\Big)\Big\vert_{x=0}
    &\mbox{if $j$ even}
    \end{cases}
    \end{multline}
    Then,  for any given positive  $N \in \mathbb N$, consider $K=2N+1$ in
    \eqref{in},
     sum and substract
     the integrals in the whole line to each term in the sum,
    estimate the differences using \eqref{e-n}, evaluate the
     integrals as in \eqref{coef} and
    estimate   the term
     with the
    remainder using  \eqref{rk}.  Rename finally  the summation
     index to obtain:
    \begin{equation}
    \label{dn}
     \mbox { for any $N\ge0$,}\quad
      D_n\,=\,\sum_{k=0}^N \frac{(-1)^k}{n^k}\,\frac{1}{2^kk!}\,
    \partial^{2k}(G^{-\frac{2k+1}{2}})(0)
    +O\big(\frac{1}{n^{N+1}}\big),
    \end{equation}
    where we have used that $1+2g(x)=G(x)$, as defined in \eqref{ak}. Let us
    shorthand \eqref{dn} above as
    \begin{equation}
    \label{cald} D_n\,\asymp\,\sum_{k\ge 0}
    \frac{(-1)^k}{n^k}\,\frac{1}{2^kk!}\,
    \partial^{2k}(G^{-\frac{2k+1}{2}})(0),
    \end{equation}
    and recall we have
    \begin{equation}
    \label{1/d} n!\,=\,\frac{\sqrt{2\pi n }\,\e^{-n} n^n }{D_n}.
    \end{equation}
    Consider next the  the well known Stirling series, which is an  asymptotic
    expansion for large
    values of $|z|$ of the logarithm of the gamma function in powers of
    $\frac 1z$
     (see for instance Theorem 1.4.2  in \cite{aar}).
     Its coefficients are given in terms of the Bernoulli
    numbers $B_k$. In
    particular, it  implies
    $$
    \log(n!)\,\asymp\, \log \big(\sqrt{2\pi n }\,\e^{-n} n^n \big)+\sum_{k\ge1}
     \frac
    {B_{2k}}{2k(2k-1)\,n^{2k-1}}
    $$
    Since the series  $\mathcal S_n$ above  contains only odd powers of
    $\frac 1n$,
     it follows that
     $$
     \log{ \big(\frac{\sqrt{2\pi n }\,\e^{-n} n^n }{n!}\big)}\,\asymp\,\,-\,
     \mathcal S_{n}=\mathcal S_{-n}.
     $$
    Therefore,
    $$
    \frac{1}{D_n}\asymp\e^{-\mathcal S_{n}}=\e^{\mathcal S_{-n}}=
    \sum_{k\ge 0}
    \frac{1}{n^k}\,\frac{1}{2^kk!}\,
    \partial^{2k}(G^{-\frac{2k+1}{2}})(0)
    $$ and the theorem follows from \eqref{1/d}
    \qed

      \vskip .8 cm
      To
    prove
    Theorem \ref{t2}, we will use the generating function of the
     $r$--associated Stirling numbers already introduced.
     Namely, (see page 222,  exercise 7, Ch. 5 of \cite{Comtet})
     : for any $r\ge 1$
     \begin{equation}
     \label{gfs}
     H(t,u):=\,\e^{u\,(\frac{t^r}{r!}+
     \frac{t^{r+1}}{(r+1)!}\,+\cdots)} \,=\,
     \sum_{l,k \ge 0}\,S_r(l,k)\,u^k\,\frac{t^l}{l!}
     \end{equation}
     \bigskip
      \noi {\it Proof of Theorem \ref{t2}}\qquad
     Observe that, from the first equation in  \eqref{ig}
      we may write the integral in the
     denominator in \eqref{n!1}
     as
     \begin{equation}
     \label{if}
     \int^{\pi\sqrt{n}}_{-\pi\sqrt{n}} d\theta\,
    E\big(\e^{i\theta \,Z_n}\big)\,=\, \int^{\pi\sqrt{n}}_{-\pi\sqrt{n}}
    d\theta\,\e^{-\frac{\theta^2}{2}} \,\e^{nf(\frac{i\theta}{\sqrt{n}})},
     \end{equation}
     with
    $f(z)=\e^z-1-z-\frac{z^2}{2}$. If we expand
    $$
    e^{nf(z)}\,=\,\sum_{j\ge 0} \frac{z^j}{j!}\,\partial^j\e^{nf}(0),
    $$
    substitute in \eqref{if} with $z=\frac{i\theta}{\sqrt{n}}$
     and proceed as in the proof of Theorem \ref{t1},
    (that is, consider a finite Taylor expansion with remainder, interchange the
    sum and
    derivative with the integral, and observe that   the integral may be
    considered in the
    whole line with an error smaller than the remainder) we obtain that for any
    given $K$,
    $$
    D_n\,=\, \sum_{j=0}^{K}\frac{i^j}{(\sqrt{n})^j\,j!}\,
    \partial^j\,\e^{nf}(0)\frac{1}{\sqrt{2\pi}}\int^{\infty}_{-\infty}
    \,d\theta\,\theta^j \e^{-\frac{\theta^2}{2}}\,+\,R_{K+1}
    $$
    The remainder can be seen to satisfy  $R_{K+1}=O(\frac{1}{n^{(K+1)/2}})$.
    After computing
    the integral  and renaming the terms, we may shorthand the above as
    \begin{equation}
    \label{dnf} D_n\,=\,
    \sum_{j\ge0}\frac{(-1)^j}{n^j\,2^j\,j!}\partial^{(2j)}\e^{nf}(0)
    \end{equation}
    Next, from the generating function for $S_3$ \eqref{gfs} and observing that
    $S_3(N,k)=0$
    if $k\ge N>0$, we obtain
    $$\partial^{(2j)}\e^{nf}(0)=\sum_{k\le 2j} n^kS_3(2j,k)
    $$
    Substituting in \eqref{dnf}, and then summing over $l=k-j$, after recalling
    that from the
    combinatorial interpretation of  $S_3$
     it is easy to see that $S_3(2j,k)=0$ if $k\ge j>0$
    $$
    D_n\,=\, \sum_{j\ge 0,\,k\le 2j}(-1)^j\frac{S_3(2j,k)}{n^{j-k}\,2^j\,j!}
    =\sum_{l\ge 0}
    \frac{(-1)^l}{n^l}\sum_{k=0}^{2l}
    \frac{(-1)^k\,S_3\big(2(k+l),k\big)}{2^{k+l}(k+l)!}
    $$
    Concluding as in the proof of \ref{t1} that $\frac{1}{D_n}=D_{-n}$, we
    obtain \eqref{c2}
    \qed \vskip .8 cm From Theorems \ref{t1} and \ref{t2}, since the  asymptotic
    expansion of
     a function
    is  unique, we have that
    \begin{equation}
    \nn
    \frac{1}{2^k\,k!}\,\partial^{2k}\big(G^{-\frac{2k+1}{2}}\big)(0)\,=\,
    \sum_{j= 0}^{2k}
    (-1)^j\,
       \frac{\,S_3\big(2(j+k),j\big)}{2^{j+k}\,(j+k)!}
    \end{equation}
    Moreover, we have a second identity that follows from the
     expression for the
    coefficients
    of the expansion for the $\gamma$ function given in \cite{Comtet}, and
    already mentioned.
    It is
    \begin{equation}
    \label{co} \Gamma(x)\,=\,x^x\e^{-x}\frac{\sqrt{2\pi}}{\sqrt x}\,
    \sum_{k\ge0}
    \frac{c_k}{x^k}, \mbox  { with } c_k=\sum_{j= 0}^{2k} (-1)^j\,
       \frac{d_3\big(2(j+k),j\big)}{2^{j+k}\,(j+k)!}.
    \end{equation}
    Permutations without
     fixed points (in other words,  without  cycles of
    length 1) are known as
     derangements, and
    for natural $r\ge1$, $d_r(n,l)$ is the number of derangements  of a set of n
    elements,
    that  have $l$ cycles, all of length $\ge r$.
     Equating the  coefficients in \eqref{c2} and the corresponding ones
    resulting
     from \eqref{co} when considering $n!=n\Gamma(n)$,
      we get the remarkable
       combinatorial identity presented in the following proposition.
     \begin{proposition}
     \label{prop}
      For $S_3$ and $d_3$ as defined above and  for any given positive
      $k\in \mathbb N$,
    \begin{equation}
    \label{id} \sum_{j= 0}^{2k} (-1)^j\,
       \frac{S_3\big(2(j+k),j\big)}{2^{j+k}\,(j+k)!}\,=\,
       \sum_{j= 0}^{2k} (-1)^j\,
       \frac{d_3\big(2(j+k),j\big)}{2^{j+k}\,(j+k)!}.
    \end{equation}
    \end{proposition}
    This identity is somehow surprising, since it is needed a precise balance to
    control the
    different
     growth of $S_3$ and $d_3$. We give a direct proof
     of a generalization of it in
     Proposition \ref{propgen}.
      Now, let us compute the generating function for $d_3(j,k)$.
      \begin{lemma}
     \label{lem}
     The generating function for  $d_3$ as defined above
     is
    \begin{equation}
    \label{gfd}
     \Phi(t,u)\,=\,(1-t)^{-u}\,\e^{-u\,(t+\frac{t^2}{2})}\,=\,\sum_{k,j\ge0}
     d_3(j,k)\,u^k\frac{t^j}{j!},
     \end{equation}
    \end{lemma}
    \noi{\it Proof }\ We first compute the exponential generating function
     of the cyclic permutations of length at least  $3$
    \begin{equation}
    \nn
      \sum_{k \ge 3}(n-1)!\,\frac{t^n}{n!}=\sum_{k\ge 3}
     \frac{t^n}{n}=\log\left(1-t\right)^{-1}-\Big(t+\frac{t^2}{2}\Big).
     \end{equation}
     By
    the exponential formula (see \cite{Stanley}, Corollary 5.1.6), we obtain
    that
    \begin{equation}
    \nn \phi(t)\,=\,(1-t)^{-1}\,\e^{-(t+\frac{t^2}{2})}
    \end{equation}
    \noindent is the exponential generating function for the number of
    derangements of a set
    of $n$ elements, whose cycles are all of order
     at least $3$. Then, the
    generating function
    \begin{equation}
    \nn \Phi(t,u)=\phi(t)^u=(1-t)^{-u}\,\e^{-u\,(t+\frac{t^2}{2})},
    \end{equation}
    \noindent also keeps track of the number of cycles on this kind of
    derangements, (see
    \cite{Stanley}, Example 5.2.2) obtaining equation (\ref{gfd}). \qed
 
    \bigskip
    By writing the generating function $H$ of $S_3$ (see \eqref{gfs}) and
     $\Phi$ in the form
    \begin{equation}
    \nn H(t,u)=\,\e^{u\,\(\,\e^{\,t}-1-t-\frac{t^2}{2}\,\)} \qquad
     \Phi(t,u)=\,\e^{u\,\(-\,\log (1-t) -t-\frac{t^2}{2}\,\)}
     \end{equation}
     the analogy  between them is apparent, and,  together with
     the identity \eqref{id}, suggests that a formula for the $a_k$ in terms of
     a {\it normalized left truncated} logarithm, in the sense
     of definition  \ref{def}, may hold. This proves to be true,
     yielding another formula, exactly alike to \eqref{asym}.
     \begin{theorem}
     \label{t3}
     The coefficients in the asymptotic expansion of $n!$
     \begin{equation}
     \nn
       n!= \sqrt{2\pi n }\,\e^{-n} n^n \Big(
     \,\sum_{k\ge 0} \frac{1}{n^k}\,a_k \Big)\quad
     \end{equation}
     \begin{align}
     \label{akl}
     &\mbox{ are also given by}\qquad
     a_k= \frac{1}{2^k\,k!}\,\partial^{2k} (L^{-\frac{2k+1}{2}})(0)\\
     \nn
     &\mbox{ for}\qquad
       L(x)\,=\,2\,\frac{-\log(1+x) \,+\,x}{x^2}\,=\, 2\sum_{j\ge 0}
       \frac{x^j}{(j+2)!}
     \end{align}
      \end{theorem}
     \noi {\it Proof}\
     Consider the identity
     $$
     \Gamma(x)\,=\, x^x\int^{\infty}_{0} \!dt\,\e^{-xt}\,
     t^{x-1}\qquad x>0
     $$
     After integrating by parts once, and changing variables
     $\sqrt x\,(t-1)\to  u$, we obtain
     \begin{align}
     \label{Ga}
     \Gamma(x)\,&=\, \frac{x^x\,\e^{-x}\sqrt{2\pi}}{\sqrt x}\Big(
     \frac{1}{\sqrt{2\pi}}\int^{\infty}_{0} \!dt\,\e^{-x(t-1)}\,\e^{x\log t}
     \sqrt{x}\Big)\\
     \nn
     &\,=\,\frac{x^x\,\e^{-x}\sqrt{2\pi}}{\sqrt x}
     \Big(\frac{1}{\sqrt{2\pi}}\int^{\infty}_{-\sqrt x} \!du\,
     \e^{-x(\frac{u}{\sqrt x})}\,\e^{x\log (\frac{u}{\sqrt x}+1)}
     \Big)\\
     \nn
     &:=\,\frac{x^x\,\e^{-x}\sqrt{2\pi}}{\sqrt x}\,I(x)
     \end{align}
     Define the function
     \begin{equation}
     \label{l}
     \ell(y)=\frac{\log(1+y)-y+y^2/2}{y^2}\,=\,\sum_{k\ge1}\frac{(-1)^{k+1}}{k+2}
     \,y^k,
     \end{equation}
     to obtain, using the expansion
     $$
     \e^{u^2\,\ell(y)}=\sum_{j\ge0}\frac{y^j}{j!}\,\partial^je^{u^2\,\ell}(0)
     $$
     in \eqref{Ga} with $y=\frac{u}{\sqrt x} $
     \begin{align}
    \nonumber I(x)\,&=\,\frac{1}{\sqrt{2\pi}}\int^{\infty}_{-\sqrt x} \!du\,
     \e^{-\frac{u^2}{2}}\,\e^{u^2\,\ell(\frac{u}{\sqrt{x}})}\\
     &=\,
     \frac{1}{\sqrt{2\pi}}\int^{\infty}_{-\sqrt x} \!du\,
     \e^{-\frac{u^2}{2}}\sum_{j\ge0}\big(\frac{u}{\sqrt{x}}\big)^j
     \frac{1}{j!}\,\partial^je^{u^2\,\ell}(0)
     \\
     \nn
     &=\,\sum_{j\ge0}\,\frac{1}{j!}\frac{1}{x^{j/2}}
     \,\partial^j
     \Big(\frac{1}{\sqrt{2\pi}}\,\int^{\infty}_{-\sqrt x} \!du\,u^j
      \,
     \e^{-\frac{u^2}{2}\big(1-2\ell(y)\big)}\Big)\big\vert_{y=0}
     \end{align}
     It is clear that
    $$
    \Big(\frac{1}{\sqrt{2\pi}}\,\int^{\infty}_{-\sqrt x} \!du\,u^j
      \,
     \e^{-\frac{u^2}{2}\big(1-2\ell(y)\big)}\Big) \approx
     \begin{cases}\big(1-2\ell(y)\big)^{-\frac{j+1}{2}}
    (j-1)!!&\mbox{ if j even }\\
    0&\mbox{ otherwise, }\end{cases}
    $$
    (in the sense that the difference is exponentially small in $x$ as $x \to
    \infty$, the
    details are analogous to that in the proof of \ref{t1}). Renaming and
    reorganizing terms
    in the sum above, and observing that $1-2\ell(y)=L(y)$ as defined
     in \eqref{akl} it follows
    $$
    I(x)\,=\,\sum_{k\ge0}\frac{1}{2^k\,k!\,x^k}\,
    \partial^{2k}\big(L^{-\frac{2k+1}{2}}\big)(0)
    $$
    Taking $x=n$ and using that $n!=n\Gamma(n)$, we conclude the proof,
     from \eqref{Ga} and the above formula for $I$.
    \qed

    \bigskip
    Since a proof of  the expansion \eqref{co}  of Comtet does not
     seem to be available, we
    provide next a
     brief sketch of it, that follows from a slight modification
     of the arguments in the
     proof above, exactly as Theorem \ref{t2} was obtained by modifying
      the proof of Theorem \ref{t1}.

     \bigskip
    \noi{\it Proof of \eqref{co}}. Write the integral in \eqref{Ga} as
    \begin{equation}
    \nonumber I(x)\,=\,
     \frac{1}{\sqrt{2\pi}}\int^{\infty}_{-\sqrt x} \!du\,
     \e^{-\frac{u^2}{2}\,+x\,h(\frac{u}{\sqrt x})},\ \mbox{ where }
     h(y)=\log{(1+y)}-y+\frac{y^2}{2}.
     \label{exph}
     \end{equation}
    and expand
    $$\e^{x\,h(y)}=\sum_{k\ge0}\frac{y^k}{k!}\,\partial^k\e^{xh}(0).
    $$
    Substituting in the integral with $y=\frac{u}{\sqrt x}$,
    \begin{align}
    \nonumber I(x)&\,=\, \frac{1}{\sqrt{2\pi}}\int^{\infty}_{-\sqrt x} \!du\,
    \e^{-\frac{u^2}{2}}\,\sum_{k\ge0}\frac{u^k}{k!\,x^{k/2}}\,
    \partial^k\e^{xh}(0)\\
    \nonumber&\, =\sum_{k\ge0}
    \frac{1}{k!\,x^{k/2}}\,\partial^k\e^{xh}(0)\,\Big(
    \frac{1}{\sqrt{2\pi}}\int^{\infty}_{-\sqrt x} \!du\,
    \e^{-\frac{u^2}{2}}\,u^k\Big)
    \end{align}
    and integrating as before
    $$
    \frac{1}{\sqrt{2\pi}}\int^{\infty}_{-\sqrt x} \!du\,
    \e^{-\frac{u^2}{2}}\,u^k \approx \begin{cases}(k-1)!!&\mbox{ if k even }\\
    0&\mbox{ otherwise, }\end{cases}
    $$
    after renaming and collecting terms, we have
    $$
    I(x)\,=\,\sum_{j\ge 0 } \frac{1}{x^j\,2^j\,j!}
    \partial^{2j}\e^{x\,h}(0)
    $$
    Now, from \eqref{gfd} and the definition of $h$, $
    \Phi(y,x)\,=\,\e^{-x\,h(-y)}$, and
    $$
    \partial^{2j}\e^{x\,h}(0)\,=\,\partial^{2j}\e^{x\,h(y)}\vert_{y=0}
    \,=\,\partial^{2j}\e^{x\,h(-y)}\vert_{y=0} \,=\,\partial^{2j} \Phi(-x,y)
    \vert_{y=0}
    $$
     Hence, from \eqref{gfd}
    $$
    I(x)\,=\,\sum_{j\ge 0 } \frac{1}{x^j\,2^j\,j!}
    \partial^{2j}\Phi(-x,y)\vert_{y=0}\,=\,\,\sum_{j\ge 0 }
    \frac{1}{x^j\,2^j\,j!} \sum_{k=0}^{2j}\,d_3(2j,k)(-x)^k,
    $$ which changing $k+l=j$ and proceeding exactly as in the
    last part of the proof of Theorem \ref{t2}, yields  \eqref{co}.
    \bigskip
    \section {Lagrange inversion and recursive formulae.}
    We will identify the ${a_k}'s$ in \eqref{ak} and \eqref{akl} as factors of the
    odd
    coefficients of a pair of formal power series, with the aid of the Lagrange
    inversion
    formula, that we recall next in a suitable form.
    Let $R(x)$ be a formal power series of the form $R(x)=xT(x)$,
     where $T(x)$ is a formal
    power series with non zero constant term. Denote by $R^{\langle -1\rangle}$
    the
    substitutional inverse of $R$, i.e., $S=R^{\langle -1\rangle}$ means
    $S(x)T(S(x))=S(xT(x))=x$. Let $S(x)=\sum_{k=1}^{\infty}s_k\frac{x^k}{k!}$.
    The Lagrange
    inversion formula gives a simple recipe to compute the coefficients of the
    series $S$
    (see for example \cite {Bergeron}).
    \begin{equation}
    \label{Lagrange}
     s_k=\partial^{k-1}\left(T^{-1}\right)^{k}(0)=
    \partial^{k-1}\left(T^{-k}\right)(0).
    \end{equation}
    Define the following exponential formal power series

    \begin{align}\label{B}
    B(x)&=\sum_{k=1}^{\infty}b_k\,\frac{x^k}{k!}:=
    \Big(x\,\big(\frac{\e^x-1-x}{x^2/2}\big)^{1/2}\Big)^{\langle
    -1\rangle}\\
    \nonumber
    &=((2\e^x-2-2x)^{1/2})^{\langle -1\rangle}\\
    \label{C}
    C(x)&=\sum_{k=1}^{\infty}c_k\,\frac{x^k}{k!}:=\Big(x\big(\frac{-\log(1+x)+x}{x^2/2}\big)^{1/2}\Big)^{\langle
    -1\rangle}\\
    \nonumber &=((-2\log(1+x)+2x)^{1/2})^{\langle -1\rangle}.
    \end{align}
     \noindent Observe that from equations
     (\ref{ak}) and (\ref{akl}), and the Lagrange
    inversion formula \eqref{Lagrange} we identify
    $$ a_k=\frac{b_{2k+1}}{2^k k!}=\frac{c_{2k+1}}{2^k k!}
    =(2k+1)!!\,\tilde{b}_{2k+1}=(2k+1)!!\,\tilde{c}_{2k+1},$$ \noindent where
    $\tilde{b}_k:=b_k/k!$ and $\tilde{c}_k:=c_k/k!$.
     Then, the coefficients
    of the formal power series $B(x)$ and $C(x)$ coincide at odd powers.
     A much stronger result holds indeed.
    \begin{proposition}
    \label{C-B} The formal power series $B(x)$
     and $C(x)$ differ only at the
    coefficient of the quadratic term. More precisely, we have
    \begin{equation}
    \nn
    C(x)-B(x)=\frac{x^2}{2}.
    \end{equation}
    \end{proposition}
   
    \begin{proof}
    From equations (\ref{B}) and (\ref{C}) we  obtain the implicit equations
    \begin{equation}\label{IB}
    \e^{B(x)}-1-B(x)=\frac{x^2}{2}
    \end{equation}
    \noindent and
    \begin{equation}\label{IC}
    C(x)-\frac{x^2}{2}=\log(1+C(x)).
    \end{equation}
    Taking $\log$ at both sides of  \eqref{IB} ,
    \begin{equation}\label{IBl}
    B(x)=\log\,\big(1+B(x)+\frac{x^2}{2}\big),
    \end{equation}
    \noindent which is equivalent to \eqref{IC} by the change
     $C(x)=B(x)+\frac{x^2}{2}$. All
    the operations to obtain respectively \eqref{IBl} and
     \eqref{IC} from \eqref{B} and \eqref{C} are reversible
      in the context of formal power series.
     Then, the result follows.
    \end{proof}
     Explicit expressions for the coefficients of the series
     $B(x)$ and $C(x)$ can be obtained
    by using Lagrange inversion formula, leading to the following generalization
    of the
    identity (\ref{id}).
    \begin{proposition}\label{propgen} The following identity holds for
     every $k\neq 2$
    \begin{multline}
    \nn
    \sum_{j=0}^{k-1}(-1)^j\frac{S_3(k+2j-1,j)}{(k+1)(k+3)\dots (k+2j-1)}\\=
    \sum_{j=0}^{k-1}(-1)^{k+j-1}\frac{d_3(k+2j-1,j)}{(k+1)(k+3)\dots (k+2j-1)}.
    \end{multline}
    \end{proposition}
    \begin{proof}
    By the Lagrange inversion formula we have
    \begin{multline}
    \nonumber
    b_k=\partial^{k-1}
    \left(\frac{\e^x-1-x}{x^2/2}\right)^{-k/2}\Big|_{x=0}\\=\partial^{k-1}
    \left(1+2\frac{\e^x-1-x-x^2/2}{x^2}\right)^{-k/2}\Big|_{x=0}
    \end{multline}
    By the binomial identity
    \begin{multline}
    \label{binom}
    \left(1+2\, \frac{\e^x-1-x-x^2/2}{x^2}\right)^{-k/2}=
    \sum_{j=0}^{\infty}\binom{-k/2}{j}2^j\,
    \frac{(\e^x-1-x-x^2/2)^j}{x^{2j}}\\
    =\sum_{j=0}^{\infty}(-1)^j k(k+2)\dots
    (k+2(j-1))\,\frac{(\e^x-1-x-x^2/2)^j}{j!x^{2j}}\\
    =\sum_{j=0}^{\infty}(-1)^j k(k+2)\dots (k+2(j-1))\frac{1}{x^{2j}}\sum_{l\geq
    3j}S_3(l,j)\frac{x^l}{l!}.
    \end{multline}
    The last equality follows from  equation (\ref{gfs}) for $r=3$, after
    writing the
    exponential in power series of the exponent, and equating powers of $u$,  to
    obtain
    \begin{equation}
    \label{es3} \frac{(\e^x-1-x-x^2/2)^j}{j!}\,=\sum_{l\ge
    3j}S_3(l,j)\frac{x^l}{l!}.
    \end{equation}
    Interchanging sums in (\ref{binom}) we get
    \begin{equation}
    \label{se}
    \sum_{l=0}^{\infty}\sum_{3j\leq l} (-1)^j \frac{k(k+2)\dots
    (k+2(j-1))}{l(l-1)\dots(l-2j+1)}S_3(l,j)\frac{x^{l-2j}}{(l-2j)!}.
    \end{equation}
    The coefficient of $x^{k-1}/(k-1)!$ of this series, obtained by making
    $l-2j=k-1$, is
    equal to
    \begin{multline}
    \nn
    b_k=\sum_{j=0}^{k-1}(-1)^j\frac{k(k+2)\dots(k+2(j-1))}{k(k+1)\dots
    (k+2j-1)}S_3(k+2j-1,j)\\=\sum_{j=0}^{k-1}(-1)^j
    \frac{S_3(k+2j-1,j)}{(k+1)(k+3)\dots
    (k+2j-1)}
    \end{multline}
    Notice that from equation \ref{gfd}, proceeding as to obtain \eqref{es3}
     $$
     \frac{(-\log(1-x)-x-x^2/2)^j}{j!}
     =\sum_{l\ge 3j}d_3(l,j)\frac{x^l}{l!}.
     $$
     From that we easily obtain
    \begin{equation}\label{gfsd}\frac{(\log(1+x)-x+x^2/2)^j}{j!}=
    \sum_{l\geq 3j}(-1)^{l-j}d_3(l,j)\frac{x^l}{l!}.\end{equation}
     By expanding the binomial
    $$ \left(1-2\frac{\log(1+x)-x+x^2/2}{x^2}\right)^{k/2}=
    \left(\frac{-\log(1+x)+x}{x^2/2}\right)^{k/2},$$ \noindent using equation
    (\ref{gfsd}),
     and proceeding exactly as for the computation of
    $b_k$, we get
    \begin{equation*}
    c_k=\sum_{j=0}^{k-1}(-1)^{k+j-1}\frac{d_3(k+2j-1,j)}{(k+1)(k+3)\dots
     (k+2j-1)}.
    \end{equation*}
    The result then follows from Proposition \ref{C-B}.
    \end{proof}
    Now we state and prove recursive formulae to obtain $b_k$, $\tilde{b}_k$,
    $c_k$,  and
    $\tilde{c}_k$.
    \begin{proposition} The sequences $b_k$, $c_k$, $\tilde{b}_k$ and
    $\tilde{c}_k$ satisfy the following recursive formulae
    \begin{align}
    \label{BR} b_k&=-\frac{1}{k+1}\left(\binom{k}{2}b_{k-1}+
    \sum_{j=1}^{k-2}\binom{k}{j}b_{j+1}b_{k-j}\right)\\\label {TBR}
    \tilde{b}_k&=-\frac{1}{k+1}\left(\frac{k-1}{2}\,
    \tilde{b}_{k-1}+\sum_{j=1}^{k-2}(j+1)\tilde{b}_{j+1}\,
    \tilde{b}_{k-j}\right)\\
    \label{CR} c_k&=\frac{1}{k+1}\left(k\,c_{k-1}-
    \sum_{j=1}^{k-2}\binom{k}{j}c_{j+1}c_{k-j}\right)\\
    \label{TCR} \tilde{c}_k&=\frac{1}{k+1}\left(\tilde{c}_{k-1}-
    \sum_{j=1}^{k-2}(j+1)\,\tilde{c}_{j+1}\tilde{c}_{k-j}\right)
    \end{align}
    \noindent with $b_1=c_1=\tilde{b}_1=\tilde{c}_1=1$.
   \end{proposition}
    \begin{proof}Formulae (\ref{TBR}) and (\ref{TCR})
     follow from formulae (\ref{BR}) and
    (\ref{CR}) respectively. Computing the derivative in both sides of equations
    (\ref{IBl})
    and (\ref{IC}), after clearing up we  obtain
    \begin{eqnarray}\label{DB}
    B'(x)B(x)&=&x-(x^2/2)B'(x)\\\label{DC} C'(x)C(x)&=&xC(x)+x.
    \end{eqnarray}
    Recurrence (\ref{BR}) (respectively (\ref{CR})) is obtained by
     equating coefficients of
    $x^k/k!$ in both sides of the resulting series in  (\ref{DB})
     (respectively (\ref{DC})).
    \end{proof}
    Observe that (\ref{TCR})  is similar to the recurrence
     obtained in \cite{Marsaglia} (see
    also \cite{Corless}).
    We list a few terms of  both formal power series $B(x)$ and C(x),
    \begin{eqnarray*}
    B(x)&=&x-x^2/6+x^3/36-x^4/270+x^5/4320+x^6/17010+\dots\\
    C(x)&=&x+x^2/3+x^3/36-x^4/270+x^5/4320+x^6/17010+\dots.
    \end{eqnarray*}
    The coefficients ${a_k}'s$ can be readily computed as well from formulae
    \eqref{ak} or \eqref{akl}, using for instance MAPLE. In particular, we checked
    that the first twenty terms coincide with that of \cite{W}.
     
    \section*{Acknowledgments}
    \nopagebreak It is our pleasure to thank Luis Baez--Duarte
     for suggesting the problem
    discussed here, and for useful conversations on the subject.
     We are grateful to Amilcar
    P\'erez for pointing out the expansion \eqref {co}.

    \end{document}